\documentclass[11pt]{amsart}

\usepackage[T1]{fontenc}
\usepackage{lmodern}
\usepackage{amsmath,amssymb,amsthm,mathtools}
\usepackage[margin=1in]{geometry}
\usepackage[colorlinks=true,linkcolor=blue,citecolor=blue,urlcolor=blue]{hyperref}

\usepackage[all]{xy}\CompileMatrices\SelectTips{cm}{12}

\renewcommand{\AA}{\mathbb A}
\newcommand{\PP}{\mathbb P}
\newcommand{\VV}{\mathbb V}
\newcommand{\cO}{\mathcal O}

\newcommand{\End}{\operatorname{End}}

\newcommand{\Hom}{\operatorname{Hom}}
\newcommand{\GL}{\operatorname{GL}}
\newcommand{\id}{\operatorname{id}}
\newcommand{\Sym}{\operatorname{Sym}}
\newcommand{\Sp}{\operatorname{Sp}}
\newcommand{\rk}{\operatorname{rk}}
\newcommand{\Proj}{\operatorname{Proj}}
\newcommand{\Spec}{\operatorname{Spec}}
\newcommand{\sslash}{\mathbin{/\mkern-6mu/}}

\newtheorem{Thm}{\sc Theorem}[section]
\newtheorem{theorem}[Thm]{\sc Theorem}
\newtheorem{corollary}[Thm]{\sc Corollary}
\newtheorem{proposition}[Thm]{\sc Proposition}
\newtheorem{lemma}[Thm]{\sc Lemma}

\theoremstyle{definition}
\newtheorem{definition}[Thm]{\sc Definition}
\newtheorem{remark}[Thm]{\sc Remark}
\newtheorem{example}[Thm]{\sc Example}

\begin{document}

\markboth {\rm }{}

\title{Ample vector bundles on non-proper schemes}
\author{Adrian Langer}

\email{alan@mimuw.edu.pl}
\address{University of Warsaw, Institute of Mathematics,
	ul.\ Banacha 2, 02-097 Warszawa, Poland}
	
\date{\today}
	
\maketitle

\begin{abstract}
We solve two open problems on ample vector bundles posed by Hartshorne in 1966. We prove that tensor products of ample vector bundles on schemes of finite type over an algebraically closed field are ample in arbitrary characteristic, extending Hartshorne’s characteristic-zero result and Barton’s projective positive-characteristic theorem.  More generally, let \(f\colon X\to S\) be a morphism of schemes. Then we prove that tensor products of
\(f\)-ample vector bundles are \(f\)-ample. Moreover, if \(E\) is an
\(f\)-ample vector bundle of rank \(r\ge 1\) and \(W\) is a finite locally
free polynomial \(\GL_{r,S}\)-module of positive rank with \(W_0=0\),
then \(E(W)\) is \(f\)-ample. In particular, \(\Gamma^nE\) is
\(f\)-ample for every \(n>0\). Finally, adapting a construction of
Ejiri--Fujino--Iwai, we show that in every characteristic there exists a smooth
quasi-projective surface that carries a non-ample extension of an ample line
bundle by itself.
\end{abstract}

\section*{Introduction}

Let \(X\) be a scheme of finite type over an algebraically closed field \(k\). A vector bundle \(E\) on \(X\) is \emph{ample} if the tautological line bundle \(\cO_{\PP_X(E)}(1)\) is ample on \(\PP_X(E)\).

Hartshorne proved that, in characteristic zero, the tensor product of
two ample vector bundles is ample
\cite[Corollary 5.3]{Hartshorne-Ample}, and asked whether the same
statement holds in positive characteristic
\cite[Introduction]{Hartshorne-Ample}. Barton subsequently answered
this question affirmatively when \(X\) is proper \cite{Barton}, leaving
open the case of non-proper schemes. Our main result settles this
remaining case and hence answers Hartshorne's question in full.

\begin{Thm}\label{main}
	Let \(X\) be a scheme of finite type over an algebraically closed field
	\(k\) of arbitrary characteristic. If \(E_1\) and \(E_2\) are ample
	vector bundles on \(X\), then \(E_1\otimes E_2\) is ample.
\end{Thm}

 Hartshorne's proof  (see \cite[Corollary 5.3]{Hartshorne-Ample}) used representation theory of $\GL (r)$. Unfortunately, in positive characteristic, this theory is notoriously complex and remains insufficiently understood, rendering generalization of his method unfeasible. On the other hand, Barton's approach \cite{Barton} used a numerical characterization of ample vector bundles on smooth projective curves, which is inapplicable in the non-proper setting. Our proof takes a completely different approach, based on Mumford’s geometric invariant theory and Seshadri’s relative version of it.
 Our method also  gives a new proof of the results of Hartshorne \cite{Hartshorne-Ample} and Barton \cite{Barton}.

The idea of the proof is as follows. It is easy to see that the above theorem follows from the fact that if 
$E$ is an ample vector bundle of rank $r\ge 2$ then $\bigwedge^2 E$ is ample.
The results of De Concini and  Procesi \cite{DeConcini-Procesi} imply that one can find a reductive group $G$
acting on the fibers of a direct sum $E^{\oplus s}$ of $s$ copies of $E$ such that
$Q:=\PP_X\bigl(\textstyle\bigwedge^2E\bigr)$ is Seshadri's relative GIT quotient of $Y:=\PP _X \bigl( E^{\oplus s}\bigr)$ over $X$ and $\cO_Q(1)$ pulls back to the restriction of $\cO_Y(2)$ to ${Y^{\mathrm{ss}}(\cO_Y(1))}$.
Since $\cO_Y(1)$ is ample we can find a $G$-equivariant locally closed immersion $j:Y\hookrightarrow \PP(V)$ into a finite dimensional projective space such that $j^*\cO_{\PP(V)}(1)$ is isomorphic to a sufficiently high tensor power of $\cO_Y(1)$. 
Then semistable points of $Y$ are mapped into semistable points of $\PP(V)$ so we have an induced morphism
$Q\to \PP(V)\sslash G$, which is finite after restricting to fibers of the projection
$Q \to X$. Since $\cO_Q(1)$ is relatively ample over $X$, this implies that $\cO_Q(1)$, and hence $\bigwedge^2 E$, are ample. The actual technical details are much more complicated as for example  we need to replace the polarization $\cO_Y(1)$ with a sufficiently large power of $\cO_Y(1)$  twisted by the pullback of the inverse of an ample line bundle on $X$.

\medskip

The same ideas, incorporating additionally van der Kallen’s power reductivity \cite{vdK} and noetherian approximation,  allow us to prove the following much stronger theorem that holds in mixed characteristic without any assumptions:

\begin{Thm}\label{main2}
	Let \(f:X\to S\) be a morphism of schemes. If \(E_1\) and \(E_2\) are $f$-ample
	vector bundles on \(X\), then \(E_1\otimes E_2\) is $f$-ample.
\end{Thm}

The absence of hypotheses on \(f\) may appear surprising. However, under our convention vector bundles have positive rank on every connected component, so existence of an \(f\)-ample vector bundle already forces
\(f\) to be quasi-compact and separated. Thus these properties of \(f\) are automatic in the statement.

If $S$ is defined over a field, the above theorem can be, with some effort,  deduced from Theorem \ref{main} by noetherian approximation. However, in general Theorem \ref{main2} is not obtained merely by applying Theorem \ref{main} to geometric fibers as fiberwise ampleness does not imply relative ampleness for a nonproper morphism
(see, e.g.,  a very simple example pointed out by Benoist in \cite{Benoist-relative-ampleness}).

\medskip

Since quotients of ample vector bundles are ample, Theorem \ref{main2} for $S=\Spec \mathbb Z$
immediately implies that, for every ample vector bundle
\(E\), the bundles \(\Sym^n E\) are ample for \(n>0\), and
\(\bigwedge^n E\) are ample for \(1\leq n\leq\rk E\). Indeed, both are
quotients of \(E^{\otimes n}\). The corresponding assertion for divided powers is subtler in positive or mixed 
characteristic as the natural map involving \(\Gamma^nE\) is an
inclusion $\Gamma^nE\to E^{\otimes n}$,
and ampleness is not in general inherited by subbundles. 
Nevertheless, with a substantial amount of additional work we are able to establish the following theorem (see Corollary \ref{representations-of-ample}).

\begin{theorem}\label{main3}
	Let \(f:X\to S\) be a morphism of schemes. 
Let \(E\) be an \(f\)-ample vector bundle of rank \(r\ge 1\), and let \(W\) be a finite locally free polynomial
\(\GL_{r,S}\)-module of positive rank whose degree zero homogeneous
component \(W_0\) vanishes. Then \(E(W)\) is \(f\)-ample.
In particular, \(\Gamma^nE\) is $f$-ample for every \(n>0\).
\end{theorem}

In principle, we obtain Theorem \ref{main3} as a corollary of Theorem \ref{main2} but the proof needs again 
 van der Kallen’s power reductivity, and existence of a certain $G$-equivariant map from divided powers to symmetric powers of a submodule (see Lemma \ref{lem:power-retraction}).

\medskip

In the same paper  Hartshorne proved that if $X$ is proper, then every extension of ample vector bundles is ample (see \cite[Corollary 3.4]{Hartshorne-Ample}) and he asked if this remained true for non-proper schemes 
(see \cite[Remark following Corollary 3.4]{Hartshorne-Ample}). 
Here we answer this question in the negative (in every characteristic, including characteristic zero).

\begin{Thm}\label{extension-of-ample}
	Let $k$ be an algebraically closed field. Then	there exists a smooth quasi-projective surface $S$ over $k$, an ample line bundle $L$ on $S$ and an extension
	\[
	0\longrightarrow L\longrightarrow E
	\longrightarrow L
	\longrightarrow 0
	\]
	such that $E$ is not ample.
\end{Thm}

Note that the theorem above implies that, without properness, ampleness need not be an open condition in families.
It also gives some further interesting examples (see Example \ref{example:funny}).

To prove the above theorem, we construct an explicit example by restricting the one given by Ejiri--Fujino--Iwai in \cite[Section 4]{Ejiri-Fujino-Iwai} to an open subscheme. Their construction provided the primary motivation for expecting such an example to exist.

\medskip

The paper is organized as follows. Section~1 collects some preliminary
results. In Section~2 we prove the key new result that the second
exterior power of an ample vector bundle is ample. Applying this result
to \(E_1\oplus E_2\) proves Theorem~\ref{main} and the case of
Theorem~\ref{main2} in which \(f\) is of finite
type over a noetherian ring; see Corollary~\ref{ample-tensor-product}. In Section~3 we use
Corollary~\ref{ample-tensor-product} to establish the ampleness of the
divided powers of an ample vector bundle. In Section~4 we combine the
preceding results with noetherian approximation to complete the proof of
Theorem~\ref{main2} and to prove Theorem~\ref{main3}; see
Corollary~\ref{representations-of-ample}. Section~5 proves
Theorem~\ref{extension-of-ample}. Finally, the appendix collects the
representation-theoretic material concerning \(\GL_r\) used in
Section~4.

\subsection*{Notation}

By a vector bundle on a scheme \(X\) we mean a finite locally free \(\mathcal O_X\)-module of positive
rank on every connected component, although the rank may vary between different connected
components. When we say that a vector bundle has rank \(r\), we mean that
it has constant rank \(r\) on \(X\).

For a quasi-coherent  \(\mathcal O_X\)-module \(E\), we set
\[
\VV_X(E)
:=
\Spec_X\bigl(\Sym_{\cO_X}E\bigr),
\qquad
\PP_X(E)
:=
\Proj_X\bigl(\Sym_{\cO_X}E\bigr).
\]
Thus if $E$ is a vector bundle then 
\(\PP_X(E)\) parametrizes rank one locally free quotients of \(E\) and
\(\cO_{\PP_X(E)}(1)\) denotes the tautological quotient line bundle.
If  \(X= \Spec R\) and $E=\widetilde M$, then we use
the notation $\VV_R(M)$  and $\PP_R(M)$ for  $\VV_X(E)$  and $\PP_X(E)$, respectively.
For a finite projective \(R\)-module \(M\) we set
\[
\mathbb A_R(M)
:=
\operatorname{Spec}_R
\bigl(\operatorname{Sym}_R(M^*)\bigr)
=
\mathbb V_R(M^*).
\]
Thus \(\mathbb A_R(M)\) is the affine \(R\)-space associated with \(M\).

Let \(X\) be an \(S\)-scheme and let \(T\to S\) be a morphism. We write
$X_T:=X\times_S T$
and, if $p_X\colon X_T\to X$ is the projection and \( E\) is
an \(\mathcal O_X\)-module, we set
$ E_T:=p_X^* E.$
Similarly, for an \(S\)-group scheme \(G\), we write
\(G_T:=G\times_S T\). If \(f\colon X\to S\) is a morphism and
\( E\) is an \(\mathcal O_S\)-module, we often abbreviate
\(f^*E\) to \( E_X\).

Let \(f\colon X\to S\) be a morphism and let \(E\) be a vector bundle
of rank \(r\) on \(X\). Its frame bundle is the right
\(\operatorname{GL}_{r,X}\)-torsor
\[
\operatorname{Fr}(E)
:=
\underline{\operatorname{Isom}}_X
\bigl(\mathcal O_X^{\oplus r},E\bigr).
\] If \(W\) is a finite locally free
\(\GL_{r,S}\)-module of positive rank on every connected component, we define the associated
vector bundle
\[
E(W)
:=
\operatorname{Fr}(E)\times^{\GL_{r,X}} W_X
\]
as the quotient of
$\operatorname{Fr}(E)\times_X W_X$ by the right action
$(p,w)\cdot g
=
\bigl(p\circ g,g^{-1}w\bigr).$

The construction \(W\mapsto E(W)\) commutes with arbitrary base change
and with direct sums, tensor products, duals, symmetric powers,
exterior powers, and divided powers.

\section{Preliminaries}

\subsection{Ample vector bundles}

Let us recall the following characterization of ample
vector bundles; see \cite[Proposition~3.2]{Hartshorne-Ample}.

\begin{proposition}\label{equivalent-definitions}
	Let \(X\) be a scheme and let \(E\) be a vector bundle on \(X\).
	The following conditions are equivalent:
	\begin{enumerate}
		\item \(E\) is ample, that is,
		\(\cO_{\PP_X(E)}(1)\) is ample;
		\item \(X\) is quasi-compact and quasi-separated, and for every
		quasi-coherent \(\cO_X\)-module \(F\) of finite type, there exists
		an integer \(n_0>0\) such that
		\[
		\Sym^nE\otimes F
		\]
		is globally generated for every \(n\geq n_0\).
	\end{enumerate}
	Moreover, in the second condition it is equivalent to let \(F\)
	range only over finitely presented quasi-coherent
	\(\cO_X\)-modules.
\end{proposition}

We shall also use the following properties; see
\cite[Proposition~2.2]{Hartshorne-Ample}.

\begin{proposition}\label{direct-sums}
	Every quotient vector bundle of an ample vector bundle on \(X\) is ample.
	Moreover, if \(E_1\) and \(E_2\) are vector bundles on \(X\), then
	\(E_1\oplus E_2\) is ample if and only if both \(E_1\) and \(E_2\) are
	ample.
\end{proposition}

Finally, we recall \cite[Corollary~2.6]{Hartshorne-Ample}.

\begin{proposition}\label{determinant}
	If \(E\) is an ample vector bundle of rank \(r\) on \(X\), then
	$\det E:=\bigwedge^r E$	is ample.
\end{proposition}

Hartshorne states these results for schemes of finite type over an
algebraically closed field; see \cite[Section~1]{Hartshorne-Ample}.
The proofs of Propositions~\ref{equivalent-definitions} and
\ref{direct-sums} extend to the present setting by using the corresponding
forms of the standard results on ample invertible sheaves and
projective bundles. Proposition~\ref{determinant} also remains valid.
In its proof, Hartshorne's representation-theoretic argument over a field
is replaced by the canonical identity
\[
\det\!\left(\Sym^n E\right)
\simeq
(\det E)^{\otimes\binom{n+r-1}{r}},
\]
which is valid for vector bundles over an arbitrary scheme.
 
\medskip
 
Now let \(f\colon X\to S\) be a morphism. Let \(E\) be a vector bundle
on \(X\) and let $\pi\colon \PP_X(E)\to X$ denote
the canonical projection. We say that \(E\) is \emph{\(f\)-ample} if 
$\cO_{\PP_X(E)}(1)$ is \((f\circ\pi)\)-ample (see 	\cite[Definition~\texttt{01VH}]{Stacks}). This property is local on \(S\). For convenience, we record the following simple fact:

\begin{lemma}\label{relatively-ample-implies-qcqs}
Let \(f\colon X\to S\) be a morphism. If there exists an $f$-ample vector bundle $E$  on $X$,
then \(f\) is quasi-compact and separated.
\end{lemma}

\begin{proof}
Since \(\cO_{\PP_X(E)}(1)\) is \((f\circ\pi)\)-ample, the morphism \(f\circ\pi\) is quasi-compact by the definition of relative ampleness  and it is separated by \cite[Lemma~{01VI}]{Stacks}. 
On the other hand, \(\pi\) is projective and surjective, so $f$ is quasi-compact.
Since  \(\pi\)	is also universally closed, separatedness descends to \(f\) by
\cite[Lemma~{09MQ}]{Stacks}.	
\end{proof}

\subsection{Relative Geometric Invariant Theory}\label{subsection:relativeGIT}

\subsubsection{Semistable loci over algebraically closed fields}
Let \(k\) be an algebraically closed field, let \(G\) be a reductive algebraic group over \(k\),
and let \(X\) be a projective \(k\)-scheme equipped with a
\(G\)-action. Suppose that \(L\) is an ample \(G\)-linearized line
bundle on \(X\). A point \(x\in X\) is called \emph{\(L\)-semistable}
(or just semistable if \(L\) is clear)
if there exist
\(n>0\) and a \(G\)-invariant section
\[
s\in H^0\bigl(X,L^{\otimes n}\bigr)^G
\]
such that \(s(x)\ne 0\). The locus of
\(L\)-semistable points is denoted by
\(X^{\mathrm{ss}}(L)\). If \(V\) is a finite-dimensional \(G\)-module, then
\(\mathcal O_{\mathbb P(V)}(1)\) carries its natural
\(G\)-linearization. In this case we abbreviate
\[
\mathbb P(V)^{\mathrm{ss}}
:=
\mathbb P(V)^{\mathrm{ss}}
\bigl(\mathcal O_{\mathbb P(V)}(1)\bigr).
\]
This is the notion of semistability that we shall use on geometric fibers.

\subsubsection{Power reductivity}

In this subsubsection we recall a few results of van der Kallen (see also \cite[Theorems~3 and~12]{FvdK} for the analogous results for Chevalley group schemes).

\begin{definition}[{see \cite[Definition 1]{vdK}}]
	Let \(R\) be a commutative ring and let \(G\) be a flat affine group scheme over \(R\). The group scheme \(G\) is called \emph{power reductive} if, for every surjective \(G\)-equivariant homomorphism $\varphi\colon M\twoheadrightarrow R$,	where \(G\) acts trivially on \(R\), there exists an integer \(d>0\) such that
	\[
	\Sym_R^d\varphi\colon\Sym_R^dM\longrightarrow\Sym_R^dR\simeq R
	\]
	admits a \(G\)-equivariant section. Equivalently, there exists
	\(t\in(\Sym_R^dM)^G\) such that $(\Sym_R^d\varphi)(t)=1.$
\end{definition}

A homomorphism \(A\to B\) of commutative \(R\)-algebras is called
\emph{power-surjective} if, for every \(b\in B\), there exists an
integer \(n>0\) such that \(b^n\) belongs to the image of \(A\to B\).

\begin{proposition}[{\cite[Proposition~7]{vdK}}]
	\label{prop:power-reductivity}
	Let \(G\) be a flat affine group scheme over a commutative ring
	\(R\). The following conditions are equivalent:
	\begin{enumerate}
		\item \(G\) is power reductive over \(R\);
		\item for every power-surjective \(G\)-equivariant homomorphism
		\(A\to B\) of commutative \(R\)-algebras, the induced homomorphism
		$A^G\longrightarrow B^G$
		is power-surjective;
		\item for every surjective \(G\)-equivariant homomorphism
		\(A\twoheadrightarrow B\) of commutative \(R\)-algebras,
		\(B^G\) is integral over the image of \(A^G\to B^G\).
	\end{enumerate}
\end{proposition}

\begin{theorem}[{\cite[Theorem~8]{vdK}}]
	\label{thm:finite-generation-invariants}
	Let \(R\) be a noetherian ring, let \(G\) be a power reductive flat
	affine group scheme over \(R\), and let \(A\) be a finitely generated
	commutative \(R\)-algebra equipped with a \(G\)-action by
	\(R\)-algebra automorphisms. Then the invariant algebra \(A^G\) is
	finitely generated over \(R\).
\end{theorem}

\begin{theorem}[{\cite[Theorem~17]{vdK}}]
	\label{thm:reductive-implies-power-reductive}	
	Let \(R\) be a commutative ring and let \(G\) be a reductive \(R\)-group scheme, i.e., a smooth affine \(R\)-group scheme whose geometric fibers are connected reductive algebraic groups.
	Then \(G\) is power reductive.
\end{theorem}

\subsubsection{Relative GIT}

Let \(S\) be a noetherian scheme and let \(G\) be a reductive
\(S\)-group scheme.
Let \(X\) be a projective \(S\)-scheme 
equipped with a \(G\)-action, and let
\(f\colon X\to S\) be its structure morphism.
Suppose that \(L\) is a \(G\)-linearized \(f\)-ample line
bundle on \(X\).

A point \(x\in X\) is called \emph{relatively \(L\)-semistable} if there
exist an affine open neighbourhood \(T\subset S\) of \(f(x)\), an integer
\(n>0\), and a section $s\in H^0\bigl(X_T,L_T^{\otimes n}\bigr)^{G_T}$
such that \(s(x)\neq0\).

The following relative form of geometric invariant theory generalizes
Seshadri's relative GIT (see
\cite[Theorem~4 and Remark~10]{Seshadri-Geometric-Reductivity}).
Seshadri's construction applies in the present setting by using power reductivity.
Although the theorem is well known to specialists (see \cite[Theorem~1.2]{Langer-Moduli-Lie-algebroids}),
we  sketch its proof for the convenience of the reader.

\begin{theorem}
	In the above setting the following hold.
	\begin{enumerate}
		\item The relatively \(L\)-semistable points form a \(G\)-stable open
		\(S\)-subscheme
		$X^{\mathrm{ss}}_S(L)\subset X.$
		For every geometric point \(\bar s\to S\), one has an equality of
		open subschemes
		\[
		\bigl(X^{\mathrm{ss}}_S(L)\bigr)_{\bar s}
		=
		X_{\bar s}^{\mathrm{ss}}(L_{\bar s}).
		\]
		\item There exist an \(S\)-scheme \(Y\) and a \(G\)-invariant affine
		surjective morphism
		$q\colon X^{\mathrm{ss}}_S(L)\longrightarrow Y$	
		such that
		\[
		\cO_Y=
		\bigl(q_*\cO_{X^{\mathrm{ss}}_S(L)}\bigr)^G.
		\]
		\item There is a projective morphism \(g\colon Y\to S\) such that	$f|_{X^{\mathrm{ss}}_S(L)}=g\circ q.$
	\end{enumerate}
\end{theorem}

\begin{proof}
	The construction is local on \(S\). Over an affine open
	\(\Spec R\subseteq S\), after replacing \(L\) by a positive power,
	which does not change the semistable locus, one writes
	\(X=\Proj_R A\) for the corresponding finitely generated
	graded \(G\)-algebra \(A\). Power reductivity gives finite generation
	of \(A^G\) and power-surjectivity for invariant sections. The affine
	quotients of the opens \(D_+(a)\), where \(a\in A^G\) is homogeneous
	of positive degree, therefore glue to
	\[
	q\colon X^{\mathrm{ss}}_S(L)\longrightarrow\Proj_R A^G.
	\]
	Power-surjectivity gives the surjectivity of \(q\), and, together
	with relative Serre vanishing, the asserted description of the
	semistable loci on geometric fibers. Finally, a suitable Veronese
	subalgebra of \(A^G\) is generated in degree one, which proves
	projectivity over \(R\). Compatibility with localization on \(R\)
	allows these local constructions to glue over \(S\).
\end{proof}

We call \(q\) the relative GIT quotient and write
$X^{\mathrm{ss}}_S(L)\sslash G$ for \(Y\) appearing in the theorem.
By \cite[Theorem~4 and Remark~9]{Seshadri-Geometric-Reductivity},
\(q\) is a uniform categorical quotient, but in general it is not a
universal categorical quotient.

If $E$ is a $G$-linearized vector bundle on $S$ then 
\(L:=\mathcal O_{\PP_S(E)}(1)\) has a natural \(G\)-linearization and the relative GIT quotient of $X:=\PP_S(E)\to S$ is constructed as
\[
q\colon X^{\mathrm{ss}}_S(L)\longrightarrow X^{\mathrm{ss}}_S(L)\sslash G:= \Proj_S\bigl((\Sym_{\cO_S}E)^G\bigr).
\]

\subsection{Fundamental theorems for the symplectic group}\label{subsection:DeCOncini-Procesi}

Let $R$ be a commutative ring and let $U$ be a free $R$-module of rank $2s$.
Let $\sigma:U\times U\to R$ be a split perfect alternating form and let
$G=\Sp(U,\sigma)$ be the corresponding symplectic group scheme over $R$.
Thus, the homomorphism
\[
\sigma^\flat:U\longrightarrow U^*,
\qquad
u\longmapsto \sigma(u,-),
\]
is an isomorphism, and $U$ admits a symplectic basis. All invariants below
are understood as group-scheme invariants.
For a finite free $R$-module $M$, we write $R[M]:=\Sym_R(M^*)$
for the coordinate ring of the associated affine $R$-scheme $\AA _R(M)$.

Let $V$ be a free $R$-module of rank $m$ and consider the natural
$G$-action on $V\otimes_R U$, where $G$ acts trivially on $V$ and through
its standard representation on $U$. Choose a basis $v_1,\ldots,v_m$ of
$V$. For every $R$-algebra $A$, every
$x\in (V\otimes_R U)_A$ can be written as
$x=v_1\otimes x_1+\cdots+v_m\otimes x_m$
for uniquely determined $x_i\in U_A$. For $1\leq i<j\leq m$, we define
a quadratic polynomial $q_{ij}\in R[V\otimes_R U]$ by
\[
q_{ij}(x):=\sigma_A(x_i,x_j).
\]
These polynomials are invariant under $G$. Let us recall the following
first and second fundamental theorems for the symplectic group, due to
C. De Concini and C. Procesi:

\begin{theorem}[{\cite[Theorems 6.6 and 6.7]{DeConcini-Procesi};
		see also \cite[Theorem 5.1]{Hashimoto}}]\label{DeConcini-Procesi-1}
	The $R$-algebra $R[V\otimes_R U]^G$ is generated by the quadratic
	forms $q_{ij}$. The kernel of the surjective homomorphism
	\[
	R[z_{ij}\mid 1\leq i<j\leq m]
	\longrightarrow R[V\otimes_R U]^G,
	\qquad
	z_{ij}\longrightarrow q_{ij},
	\]
	is generated by the Pfaffians of principal submatrices of size
	$2s+2$ of the alternating matrix $[a_{ij}]$ given by
	\[
	a_{ii}=0,\qquad
	a_{ij}=z_{ij}\ \text{for }i<j,\qquad
	a_{ij}=-z_{ji}\ \text{for }i>j.
	\]
\end{theorem}

The theorem is valid without  assuming that $2$ is invertible in
$R$. In particular, it also holds in characteristic $2$, where
``alternating'' includes the condition that all diagonal entries vanish.

Moreover, the above presentation is compatible with arbitrary base
change (even though formation of invariants does not commute with
arbitrary base change in general). This means that for every $R$-algebra $R'$, the canonical map
\[
R[V\otimes_R U]^G\otimes_R R'
\longrightarrow
R'[V_{R'}\otimes_{R'}U_{R'}]^{G_{R'}}
\]
is an isomorphism, where
$V_{R'}=V\otimes_R R'$, and similarly for $U_{R'}$ and $G_{R'}$.
Indeed, on both sides Theorem \ref{DeConcini-Procesi-1} identifies this
map with the base change of the same Pfaffian algebra over $\mathbb Z$.

\begin{corollary}\label{DeConcini-Procesi-3}
	Assume that $\operatorname{rk}_R V=m\leq 2s+1$. Let
	$
	\mu\colon	\mathbb A_R(V\otimes_RU)
	\longrightarrow
	\mathbb A_R\bigl(\textstyle\bigwedge_R^2V\bigr)$
	be the \(G\)-invariant morphism whose value on \(A\)-points is given by
	\[
	\mu(x)=
	\sum_{1\leq i<j\leq m}q_{ij}(x)\,v_i\wedge v_j.
	\]
	Equivalently, regarding $x$ as a homomorphism $V^*\to U$, the element
	$\mu(x)$ is the alternating form on $V^*$ obtained by pulling back
	$\sigma$. In particular, the morphism $\mu$ is independent of the
	chosen basis of $V$.
	Then the induced map
	\[
	\AA_R(V\otimes_R U)\sslash G
	:=
	\Spec\bigl(R[V\otimes_R U]^G\bigr)
	\longrightarrow \AA_R(
	\bigwedge\nolimits_R^2V)
	\]
	is an isomorphism of affine $R$-schemes.
\end{corollary}

Let us choose a symplectic basis
$e_1,\ldots,e_s,f_1,\ldots,f_s$ of $U$, i.e., such that
\[
\sigma(e_i,e_j)=\sigma(f_i,f_j)=0,
\qquad
\sigma(e_i,f_j)=\delta_{ij},
\]
and consider the map
\[
\varphi:
\Sym_R\bigl(\textstyle\bigwedge_R^2V\bigr)
\longrightarrow
\bigl(\Sym_R(V\otimes_R U)\bigr)^G
\]
given on generators by
\[
v\wedge w\longmapsto
\sum_{i=1}^{s}
\bigl(
(v\otimes e_i)(w\otimes f_i)
-
(v\otimes f_i)(w\otimes e_i)
\bigr).
\]
If we place $\bigwedge_R^2V$ in degree $2$ and $V\otimes_R U$ in
degree $1$, then $\varphi$ becomes a homomorphism of graded
$R$-algebras.

Let us recall that there is a natural $G$-equivariant map
\[
\psi:
\bigwedge\nolimits_R^2V\otimes_R\bigwedge\nolimits_R^2U
\longrightarrow
\Sym_R^2(V\otimes_R U)
\]
given by
\[
(v\wedge w)\otimes(u\wedge z)
\longmapsto
(v\otimes u)(w\otimes z)
-
(v\otimes z)(w\otimes u).
\]
The tensor
\[
\pi_\sigma:=
\sum_{i=1}^{s}e_i\wedge f_i
\in\bigwedge\nolimits_R^2U
\]
is independent of the choice of symplectic basis and it is
$G$-invariant. Indeed, if $\sigma$ is regarded as an element of
$\bigwedge_R^2U^*$, then
$\bigl(\textstyle\bigwedge^2\sigma^\flat\bigr)(\pi_\sigma)=\sigma.$
Thus, applying $\psi$ to
\[
(v\wedge w)\otimes\pi_\sigma
=
\sum_{i=1}^{s}(v\wedge w)\otimes(e_i\wedge f_i),
\]
we recover the map
\[
\bigwedge\nolimits_R^2V
\longrightarrow
\Sym_R^2(V\otimes_R U)^G
\]
inducing $\varphi$.

Applying Corollary \ref{DeConcini-Procesi-3} with $V^*$ in place of $V$,
and using the canonical identifications
\[
(V^*)^*\simeq V,
\qquad
\bigl(\textstyle\bigwedge_R^2V^*\bigr)^*
\simeq
\bigwedge\nolimits_R^2V,
\]
together with the $G$-equivariant isomorphism
\[
U\xrightarrow{\sim}U^*,
\qquad
u\longmapsto\sigma(u,-),
\]
gives the following equivalent formulation of Corollary
\ref{DeConcini-Procesi-3}.

\begin{corollary}\label{DeConcini-Procesi-2}
	Assume that $\operatorname{rk}_R V\leq 2s+1$. Then $\varphi$ is an
	isomorphism of graded $R$-algebras. In particular, for every $n\geq0$
	the induced maps
	\[
	\Sym_R^n\bigl(\textstyle\bigwedge_R^2V\bigr)
	\longrightarrow
	\bigl(\Sym_R^{2n}(V\otimes_R U)\bigr)^G
	\]
	are isomorphisms and
	\[
	\bigl(\Sym_R^{2n+1}(V\otimes_R U)\bigr)^G=0.
	\]
\end{corollary}

The construction of \(\varphi\) is functorial in \(V\). Moreover, the
resulting identification of the invariant algebra is compatible with
arbitrary base change. The latter assertion is nontrivial, since, as remarked before,  the
formation of invariants does not, in general, commute with arbitrary
base change; see Subsection~\ref{subsection:relativeGIT}.
Consequently, if $X$ is an $R$-scheme, and $E$ is a
locally free $\mathcal O_X$-module of rank at most $2s+1$,
then the local isomorphisms of Corollary
\ref{DeConcini-Procesi-2} glue to a canonical isomorphism
\[
\Sym_{\mathcal O_X}
\bigl(\textstyle\bigwedge_{\mathcal O_X}^2E\bigr)
\xrightarrow{\ \sim\ }
\left(
\Sym_{\mathcal O_X}
(E\otimes_{R}U)
\right)^{G}, 
\]
where $\bigwedge_{\mathcal O_X}^2E$ is placed in degree $2$.

\subsection{Divided powers and power retractions}
\label{subsection:divided-powers}

Throughout this subsection, \(R\) is a commutative ring and \(W\) is a
finitely generated projective \(R\)-module.

Let $\Gamma(W)=\bigoplus_{d\geq0}\Gamma^dW$
denote the divided-power algebra of \(W\). Its homogeneous component
\(\Gamma^dW\) represents homogeneous polynomial laws of degree \(d\)
from \(W\). The universal homogeneous polynomial law of degree \(d\)
is denoted by
\[
\gamma_d\colon W\longrightarrow\Gamma^dW,
\qquad
w\longmapsto\gamma_d(w).
\]
Thus, for every \(R\)-module \(V\), composition with \(\gamma_d\)
induces a natural bijection between \(R\)-linear maps
\(\Gamma^dW\longrightarrow V\)
and homogeneous polynomial laws \(W\to V\) of degree \(d\).

Since \(W\) is finitely generated projective, there are natural
identifications
\[
\Gamma^dW
\simeq
\bigl(\Sym_R^d(W^*)\bigr)^*
\simeq
\bigl(W^{\otimes_R d}\bigr)^{\mathfrak S_d},
\]
under the second of which \(\gamma_d(w)\) corresponds to \(w^{\otimes d}\).

Formation of \(\Gamma^dW\), as well as the universal law \(\gamma_d\),
commutes with arbitrary base change. Consequently, if \(E\) is a vector
bundle on a scheme \(X\), the corresponding local constructions glue to
a vector bundle \(\Gamma^dE\), and there is a natural identification
\[
\Gamma^dE
\simeq
\bigl(\Sym_{\mathcal O_X}^d(E^*)\bigr)^*.
\]

If \(G\) is a group scheme over \(R\) and \(W\) is a \(G\)-module, then
functoriality gives \(\Gamma^dW\) a natural \(G\)-module structure; see
\cite[Part~I, 2.7 and 2.8]{Jantzen}.

\medskip

The following lemma plays a crucial role in the proof of
Theorem~\ref{thm:divided-powers-ample}.

\begin{lemma}
	\label{lem:power-retraction}
	Let \(G\) be a reductive group scheme over \(R\), and let
	\(U\subset W\) be a \(G\)-submodule. Assume that \(U\) and \(W/U\)
	are finitely generated projective \(R\)-modules and that \(U\) has
	positive rank at every point of \(\Spec R\).
	Then there exist an integer \(d>0\) and a \(G\)-equivariant
	\(R\)-linear map
	\[
	\vartheta\colon
	\Gamma^dW\longrightarrow\Sym_R^dU
	\]
	such that, for every \(R\)-algebra \(A\),
	\[
	\vartheta_A\bigl(\gamma_{d,A}(u)\bigr)=u^d
	\qquad\text{for every }u\in U_A.
	\]
	Here
	\(\gamma_{d,A}\) denotes the base change of \(\gamma_d\), and
	\(\vartheta_A\) denotes the base change of \(\vartheta\).
\end{lemma}

\begin{proof}
	Let $	M:=\Hom_R(W,U)$	with the conjugation action	$(g\cdot\varphi)(w)=	g\varphi(g^{-1}w).$
	Set
	$M_0	:=
	\left\{
	\varphi\in M
	\;\middle|\;
	\varphi|_U\in R\,\id_U
	\right\}.$
	This is a \(G\)-submodule of \(M\). Since \(U\) has positive rank at
	every point of \(\Spec R\), the homomorphism
	\[
	R\longrightarrow\End_R(U),
	\qquad
	a\longmapsto a\,\id_U,
	\]
	is injective. Consequently, there is a uniquely determined
	\(R\)-linear map	$\ell\colon M_0\longrightarrow R$
	such that
	\[
	\varphi|_U=\ell(\varphi)\id_U
	\qquad\text{for every }\varphi\in M_0.
	\]
	The map \(\ell\) is \(G\)-equivariant, where \(G\) acts trivially
	on \(R\).
	
	Since \(W/U\) is projective, there exists an
	\(R\)-linear map	$\rho\colon W\longrightarrow U$ such that
 \(\rho|_U=\id_U\). Hence \(\rho\in M_0\) and \(\ell(\rho)=1\).
	Moreover, restriction to \(U\) gives a split exact sequence
	\[
	0\longrightarrow\Hom_R(W/U,U)
	\longrightarrow M_0
	\xrightarrow{\ \ell\ }R
	\longrightarrow0,
	\]
	where a splitting is given by \(a\mapsto a\rho\). Equivalently,
	\[
	M_0
	\simeq
	\Hom_R(W/U,U)\oplus R\rho
	\]
	as \(R\)-modules. In particular, \(M_0\) is finitely generated
	projective, and its formation commutes with arbitrary base change.
	
	By Theorem \ref{thm:reductive-implies-power-reductive}, there exist an
	integer $d>0$ and an element
	$t\in \bigl(\operatorname{Sym}_R^dM_0\bigr)^G$
	such that
	$	(\operatorname{Sym}_R^d\ell)(t)=1.$
	There is a canonical $G$-equivariant $R$-linear map
	\[\Sym_R^dM=
	\operatorname{Sym}_R^d\operatorname{Hom}_R(W,U)
	\longrightarrow
	\operatorname{Hom}_R\bigl(\Gamma^dW,\operatorname{Sym}_R^dU\bigr),
	\]
	which sends $\varphi_1\cdots\varphi_d$ to the linearization, through the
	universal property of $\Gamma^dW$, of the homogeneous polynomial law
	$w\longmapsto \varphi_1(w)\cdots\varphi_d(w).$
	The inclusion $M_0\subset M$ induces a homomorphism
	\[
	\operatorname{Sym}_R^dM_0
	\longrightarrow
	\operatorname{Sym}_R^dM.
	\]
	Let $\vartheta$ be the image of $t$ under the composition of these
	two maps. Since $t$ is $G$-invariant,
	$\vartheta$ is $G$-equivariant. For every $R$-algebra $A$, every
	$u\in U_A$, and every $\varphi\in (M_0)_A$, we have
	$\varphi(u)=\ell_A(\varphi)u.$
	Therefore
	\[
	\vartheta_A\bigl(\gamma_{d,A}(u)\bigr)
	=(\operatorname{Sym}_A^d\ell_A)(t_A)u^d
	=u^d.
	\]
\end{proof}

\subsection{Noetherian approximation}
\label{subsection:noetherian-approximation}

We will need the following noetherian approximation statement to remove the noetherian and finite type hypotheses from Corollaries~\ref{ample-tensor-product} and \ref{representations-of-ample-1}  (see Corollary \ref{representations-of-ample}).

\begin{lemma}\label{approximation-of-ample-vector-bundles}
	Let	$X\simeq\varprojlim_{\lambda\in\Lambda}X_\lambda$
	be an inverse limit of quasi-compact and quasi-separated schemes with
	affine transition morphisms, and let
	\(g_\lambda\colon X\to X_\lambda\) be the canonical projections.
	Then:
	\begin{enumerate}
		\item Every vector bundle \(E\) on \(X\) is isomorphic to
		\(g_\lambda^*E_\lambda\) for some \(\lambda\) and some vector bundle
		\(E_\lambda\) on \(X_\lambda\).
		\item If \(E\) is ample, then	\(E_\lambda\) can be chosen ample.
		\item The pullback to \(X\) of an ample vector bundle on some \(X_\lambda\) is ample.
	\end{enumerate}
\end{lemma}

\begin{proof}
	(1) is the content of \cite[Lemma~\texttt{0B8W}]{Stacks}. Let us choose an index
	\(\lambda\) and a finite locally free
	\(\cO_{X_\lambda}\)-module \(E_\lambda\) such that
	$E\simeq g_\lambda^*E_\lambda.$
	Let \(U_\lambda\subset X_\lambda\) be the open and closed locus on
	which \(E_\lambda\) has positive rank. Since \(E\) is a vector bundle
	in our convention, \(g_\lambda(X)\subset U_\lambda\). After increasing
	\(\lambda\), we may assume that \(E_\lambda\) has positive rank
	everywhere; see \cite[Lemma~\texttt{05F4}]{Stacks}. Thus
	\(E_\lambda\) is a vector bundle in our sense.
	For \(\mu\geq\lambda\), denote the transition morphism by
	\(g_{\mu\lambda}\colon X_\mu\to X_\lambda\), and set
	\[
	E_\mu:=g_{\mu\lambda}^*E_\lambda,
	\qquad
	Y_\mu:=\PP_{X_\mu}(E_\mu),
	\qquad
	L_\mu:=\cO_{Y_\mu}(1).
	\]
	We also set	$Y:=\PP_X(E)$ and $L:=\cO_Y(1).$
	Formation of the relative projective bundle and of its tautological
	line bundle commutes with arbitrary base change
	\cite[Lemma~\texttt{01O3}]{Stacks}. Consequently,
	\[
	Y\simeq\varprojlim_{\mu\geq\lambda}Y_\mu,
	\qquad
	L\simeq h_\mu^*L_\mu,
	\]
	where \(h_\mu\colon Y\to Y_\mu\) is the canonical projection. Each
	\(Y_\mu\) is quasi-compact and quasi-separated, and the transition
	morphisms in the inverse system \((Y_\mu)_{\mu\geq\lambda}\) are
	affine, since they are base changes of the corresponding transition
	morphisms in \((X_\mu)_{\mu\geq\lambda}\).
	
	Suppose that \(E\) is ample. Then \(L\) is ample, so
	\cite[Lemma~\texttt{09MT}]{Stacks} shows that \(L_\mu\) is ample for
	some \(\mu\geq\lambda\). Hence \(E_\mu\) is ample, proving (2).
	
To prove (3) suppose that \(E_\lambda\) is ample. The morphism
	\(g_\lambda\colon X\to X_\lambda\) is affine by
	\cite[Lemma~\texttt{01YX}]{Stacks}. Consequently,
	\(h_\lambda\colon Y\to Y_\lambda\) is affine, since it is obtained
	from \(g_\lambda\) by base change. Since
	$	L\simeq h_\lambda^*L_\lambda$,
	the ampleness of \(L\) follows from
	\cite[Lemma~\texttt{0892}]{Stacks}. Thus \(E\) is ample.
\end{proof}

\section{Second exterior power of ample vector bundles}

\begin{theorem}\label{ample-second-wedge}
	Let $X$ be a scheme of finite type over a noetherian ring $R$, and let $E$ be an ample vector bundle of rank $r\ge 2$ on $X$. Then $\bigwedge^2 E$ is ample.
\end{theorem}

\begin{proof}

	\emph{Step 1: The relative version of fundamental theorems for the symplectic group.}
	
	Choose a split symplectic $R$-module $(U,\sigma)$ of rank $2s\geq r-1$, and put
	\[
	W:=E\otimes_R U\simeq E^{\oplus 2s}.
	\]
	$W$ is ample by Proposition \ref{direct-sums}. So the line bundle $L:=\cO_{\PP_X(W)}(1)$ is ample on $Y:=\PP_X(W)$.
	
	Let $G=\Sp_R(U,\sigma)$ be the corresponding reductive group scheme over $R$.
	We denote its base change to $X$ again by $G$. This group scheme acts on $Y$ so that the vector bundle $W$ has a canonical $G$-linearization, and the canonical projection $f: Y\to X$ is $G$-invariant.
	So we can consider the relative projective GIT quotient
	\[
	q\colon Y^{o}:=Y^{\mathrm{ss}}_X(L)\longrightarrow Y^{\mathrm{ss}}_X(L)\sslash G= \Proj_X\bigl((\Sym W)^G\bigr)
	\]
	over $X$. Let us consider
	$g\colon Q:=\PP_X\bigl(\textstyle\bigwedge^2E\bigr) \to X$ and  set $M:=\cO_{\PP_X\bigl(\textstyle\bigwedge^2E\bigr)}(1)$. 
	Corollary \ref{DeConcini-Procesi-2}, applied locally on $X$, implies that the canonical homomorphism
	\[
	\Sym\bigl(\textstyle\bigwedge^2E\bigr) \to \bigl(\Sym W\bigr)^G 
	\]
	is an isomorphism of appropriately graded $\cO_X$-algebras, where $\bigl(\Sym W\bigr)^G$ vanishes in odd  degrees. These local isomorphisms are compatible with changes of trivialization of $E$ and hence glue.
	Therefore $Y^{\mathrm{ss}}_X(L)\sslash G$ can be identified with $Q$. The quotient map $q$ is surjective, and we have 
	a canonical isomorphism
	\begin{equation}
		\label{pullback-of-M}
		q^*M\simeq L^2|_{Y^{o}}
	\end{equation}
	of $G$-linearized line bundles (we consider $q^*M$ with the natural pullback linearization).
	
	\medskip
	\noindent
	\emph{Step 2: Absolute ampleness of $M$.}
	
	The line bundle $M$ is $g$-ample. Choose an ample line bundle $H$ on $X$ (such a line bundle exists as $\det E$ is ample by Proposition \ref{determinant}). Let us also choose a positive
	integer $a>0$ such that
	\[
	A:=M\otimes g^*H^a
	\]
	is ample on $Q$. By \cite[Corollaire 4.5.8]{EGA2} we can choose  $N>1$ so that
	\[
	B_0:=L^{2(N-1)}\otimes f^*H^{-a}
	\]
	is ample on $Y$. Let us  set  $D_0:=M^N\otimes A^{-1}$.
	Then by (\ref{pullback-of-M}) we have 
	an isomorphism of $G$-linearized  line bundles 
	$
	q^*D_0\simeq B_0|_{Y^{o}}.$
	Let $t>0$ be such that $B:=B_0^t$ is very ample over $\Spec R$, and set
	$D:=D_0^t$ so that we get an isomorphism of $G$-linearized  line bundles 
	\begin{equation}
		\label{pullback-of-D}
		q^*D\simeq B|_{Y^{o}}.	
	\end{equation}
	
	Choose finitely many sections of $B$ defining a locally closed immersion of
	$Y$ into a projective space over $R$. Since $R$ is noetherian, every
	$G$-module is the union of its $G$-submodules which are finitely generated
	over $R$ (see \cite[Proposition 2]{Ser68}). Hence these sections are contained
	in a finitely generated $G$-stable $R$-submodule
	$V\subseteq H^0(Y,B).$
	The corresponding linear system defines a $G$-equivariant locally closed immersion
	$j\colon Y\to\PP_R(V)$
	with an isomorphism $j^*\cO_{\PP_R(V)}(1)\simeq B$ of $G$-linearized line bundles
	(this argument is a generalization of the proof of \cite[Chapter 1, Proposition 1.7]{Mumford-Fogarty-Kirwan}).	
  
\begin{lemma}\label{semistability}
	For every geometric point \(\bar y\) of \(Y^{o}\), there exist an
	integer \(d_{\bar y}>0\) and an invariant
	\[
	T_{\bar y}\in
	\bigl(\Sym_R^{d_{\bar y}}V\bigr)^G
	\]
	such that
$T_{\bar y}\bigl(j(\bar y)\bigr)\neq 0.$
\end{lemma}

\begin{proof}
	Let \(\bar y\colon\Spec\Omega\to Y^{o}\) be a geometric point and
	set \(\bar x=f(\bar y)\), where \(\Omega\) is algebraically closed.
	The restriction of the base change of \(j\) gives a
	\(G_\Omega\)-equivariant morphism
	\[
	j_{\bar x}\colon
	Y_{\bar x}=\PP(W_{\bar x})
	\longrightarrow \PP_\Omega(V_\Omega).
	\]
	It is a closed immersion by \cite[Lemma 01W6]{Stacks}. Indeed, it is a locally closed immersion,
	and \(Y_{\bar x}\) is proper over \(\Omega\), whereas
	\(\PP_\Omega(V_\Omega)\) is separated over \(\Omega\).
	
	Since
	\(\bar y\in Y^{o}=Y_X^{\mathrm{ss}}(L)\), the point \(\bar y\) is
	\(L|_{Y_{\bar x}}\)-semistable; see Subsection
	\ref{subsection:relativeGIT}. Moreover,
	\[
	B|_{Y_{\bar x}}
	\simeq
	L^{2t(N-1)}|_{Y_{\bar x}}
	\otimes_\Omega H_{\bar x}^{-at},
	\]
	where \(H_{\bar x}^{-at}\) is a one-dimensional
	\(\Omega\)-vector space with trivial \(G_\Omega\)-action. Hence
	\(\bar y\) is also \(B|_{Y_{\bar x}}\)-semistable. Since
	\[
	j_{\bar x}^*\cO_{\PP_\Omega(V_\Omega)}(1)
	\simeq B|_{Y_{\bar x}}
	\]
	as \(G_\Omega\)-linearized line bundles, the Hilbert--Mumford
	criterion shows that \(j_{\bar x}(\bar y)\) is semistable in
	\(\PP_\Omega(V_\Omega)\) (see also
	\cite[Chapter~1, Theorem~1.19 and the discussion after
	Corollary~1.20]{Mumford-Fogarty-Kirwan} for another argument).
	
		Set
	$\mathcal S_V:=\Sym_R V$
	and let \(\mathfrak p\in\Spec R\) be the image of \(\bar x\).
	By semistability, there exist an integer \(d_0>0\) and an invariant
	\[
	\overline{\tau}\in
	\bigl(\mathcal S_V\otimes_R\Omega\bigr)^{G_\Omega}_{d_0}
	\]
	such that
	$	\overline{\tau}\bigl(j_{\bar x}(\bar y)\bigr)\neq 0.$
	Flat base change for invariants gives an isomorphism
	\[
	\bigl(
	\mathcal S_V\otimes_R\kappa(\mathfrak p)
	\bigr)^{G_{\kappa(\mathfrak p)}}_{d_0}
	\otimes_{\kappa(\mathfrak p)}\Omega
	\xrightarrow{\ \sim\ }
	\bigl(
	\mathcal S_V\otimes_R\Omega
	\bigr)^{G_\Omega}_{d_0}.
	\]
	Writing \(\overline{\tau}\) as an \(\Omega\)-linear combination of
	base changes of invariants over \(\kappa(\mathfrak p)\), we find
	an element
	\[
	\tau\in
	\bigl(
	\mathcal S_V\otimes_R\kappa(\mathfrak p)
	\bigr)^{G_{\kappa(\mathfrak p)}}_{d_0},
	\]
	whose base change to \(\Omega\) does not vanish at
	\(j_{\bar x}(\bar y)\).
	
	The natural homomorphism of graded \(G_{R_{\mathfrak p}}\)-algebras
	\[
	(\mathcal S_V)_{\mathfrak p}
	\longrightarrow
	(\mathcal S_V)_{\mathfrak p}
	\otimes_{R_{\mathfrak p}}\kappa(\mathfrak p)
	\]
	is surjective. Since \(G_{R_{\mathfrak p}}\)  is power reductive by Theorem \ref{thm:reductive-implies-power-reductive}, 
	Proposition \ref{prop:power-reductivity} implies that there exist an integer \(m>0\) and
	an invariant
	\[
	\widetilde T\in
	\bigl((\mathcal S_V)_{\mathfrak p}^{\,G}\bigr)_{md_0}
	\]
	whose image is \(\tau^m\). Here we have taken the homogeneous
	component of degree \(md_0\) of an invariant lift.
	
	Invariants commute with localization, so
	$	(\mathcal S_V^G)_{\mathfrak p}
	=
	\bigl((\mathcal S_V)_{\mathfrak p}\bigr)^G.$
	Consequently, after clearing a denominator, there exist
	\(a\in R\setminus\mathfrak p\) and
	\[
	T_{\bar y}\in
	\bigl(\Sym_R^{md_0}V\bigr)^G
	\]
	such that \(\widetilde T=T_{\bar y}/a\). Since \(a\) is nonzero at
	\(\mathfrak p\), we obtain
	\[
	T_{\bar y}\bigl(j(\bar y)\bigr)
	=
	a(\bar y)\,
	\tau\bigl(j_{\bar x}(\bar y)\bigr)^m
	\neq 0.
	\]
	The assertion follows with \(d_{\bar y}=md_0\).
\end{proof}

\begin{lemma}\label{globally-generated}
	There exists an integer \(e>0\) such that \(D^e\) is globally
	generated.
\end{lemma}

\begin{proof}
	Fix a point \(z\in Q\), choose a geometric point
	\(\bar z\to Q\) lying over \(z\), and, using the surjectivity of
	\(q\), choose a geometric point $\bar y\in Y^{o}_{\bar z}.$
	By Lemma \ref{semistability}, there exist \(d_{\bar y}>0\) and
	\[
	T_{\bar y}\in
	\bigl(\Sym_R^{d_{\bar y}}V\bigr)^G
	\]
	such that \(T_{\bar y}(j(\bar y))\neq0\).	
	Let
	\[
	\widetilde T_{\bar y}\in
	H^0\bigl(\PP_R(V),\cO_{\PP_R(V)}(d_{\bar y})\bigr)^G
	\]
	be the image of \(T_{\bar y}\) under the canonical homomorphism
	\[
	\Sym_R^{d_{\bar y}}V
	\longrightarrow
	H^0\bigl(\PP_R(V),
	\cO_{\PP_R(V)}(d_{\bar y})\bigr).
	\]
	Pulling this section back by \(j\) gives
	\[
	s_{\bar y}:=j^*\widetilde T_{\bar y}
	\in H^0\bigl(Y,B^{d_{\bar y}}\bigr)^G,
	\qquad
	s_{\bar y}(\bar y)\neq0.
	\]
	By \eqref{pullback-of-D}, the restriction
	\(s_{\bar y}|_{Y^{o}}\) is a \(G\)-invariant section of
	\(q^*D^{d_{\bar y}}\).
	
	Since \(q\) is a good quotient, the projection formula gives
	\[
		\bigl(q_*q^*D^{d_{\bar y}}\bigr)^G
		\simeq
		\bigl(D^{d_{\bar y}}\otimes q_*\cO_{Y^{o}}\bigr)^G 
		\simeq
		D^{d_{\bar y}}
		\otimes\bigl(q_*\cO_{Y^{o}}\bigr)^G 
		\simeq D^{d_{\bar y}}.
	\]
	Thus \(s_{\bar y}|_{Y^{o}}\) descends uniquely to a section
	$\overline{s}_z\in H^0\bigl(Q,D^{d_{\bar y}}\bigr).$
	This section is nonzero at \(\bar z\), and hence at \(z\).
	
	Let \(U_z\subseteq Q\) be the nonvanishing locus of
	\(\overline{s}_z\). Since \(Q\) is of finite type over the
	noetherian ring \(R\), it is quasi-compact. We may therefore choose
	finitely many points \(z_1,\ldots,z_m\) such that
	\[
	Q=U_{z_1}\cup\cdots\cup U_{z_m}.
	\]
	Let \(d_i\) be the degree of \(\overline{s}_{z_i}\), and let \(e\)
	be a common multiple of \(d_1,\ldots,d_m\). The sections
	\[
	\overline{s}_{z_i}^{\,e/d_i}\in H^0(Q,D^e),
	\qquad 1\leq i\leq m,
	\]
	have no common zero. Hence \(D^e\) is globally generated.
\end{proof}

Finally, by the definition of \(D\), we have
\[
M^{Nte}=A^{te}\otimes D^e.
\]
The first factor on the right is ample, while the second is globally
generated by Lemma \ref{globally-generated}. Their tensor product is
therefore ample; see, for example,
\cite[Proposition~4.5.6]{EGA2}. Thus a positive power of \(M\) is
ample, and consequently \(M\) itself is ample. Since
\[
M=\cO_{\PP_X(\bigwedge^2E)}(1),
\]
the vector bundle \(\bigwedge^2E\) is ample.
\end{proof}

\begin{corollary}\label{ample-tensor-product}
	Let \(X\) be a scheme of finite type over a noetherian ring \(R\). If
	\(E_1\) and \(E_2\) are ample vector bundles on \(X\), then
	\(E_1\otimes E_2\) is ample.
\end{corollary}

\begin{proof} Without loss of generality we can assume that $X$ is connected.
	By Proposition~\ref{direct-sums}, the vector bundle \(E_1\oplus E_2\) is
	ample. Hence Theorem~\ref{ample-second-wedge} shows that
	\[
	\bigwedge\nolimits^2(E_1\oplus E_2)
	\simeq
	\bigwedge\nolimits^2E_1
	\oplus(E_1\otimes E_2)
	\oplus\bigwedge\nolimits^2E_2
	\]
	is ample. Since \(E_1\otimes E_2\) is a quotient of this bundle, it is
	ample by Proposition~\ref{direct-sums}.
\end{proof}

\medspace
\begin{remark}\label{failed-attempt}
	When \(R\) is an algebraically closed field---the case originally
	considered in this paper---Corollary~\ref{ample-tensor-product} can
	also be deduced from an analogue of
	Theorem~\ref{ample-second-wedge} for the second divided power.
	This analogue can be proved by adapting the proof of
	Theorem~\ref{ample-second-wedge}, using the fundamental theorems of
	invariant theory for the special orthogonal group in place of those
	for the symplectic group. The resulting argument is, however, more
	involved, since the cases of characteristic \(2\) and characteristic
	different from \(2\) must be treated separately. In characteristic
	\(2\), one must use the results of Domokos and
	Frenkel~\cite[Theorem~4.9]{DomokosFrenkel}. At present, this approach
	does not appear to extend to mixed characteristic.
\end{remark}

\section{Ampleness of higher divided powers}
\label{sec:higher-divided-powers}

Throughout this section, $R$ is a noetherian ring and $X$ is a scheme of finite type over $R$.

\medskip

\begin{theorem}
	\label{thm:divided-powers-ample}
	Let $E$ be an ample vector bundle on $X$.    Then  \(\Gamma^nE\) 
	is ample for every $n\geq 1$.
\end{theorem}

\begin{proof} 	
Fix $n\geq 1$. Without loss of generality we can assume that $E$ has constant rank $r\ge 1$. We set
	\[V:=R^r,\qquad	G:=\GL(V),
	\qquad	U:=\Gamma^nV,\qquad
	W:=V^{\otimes n}.	\]
	Since $\Gamma^nV\simeq (V^{\otimes n})^{\mathfrak S_n}$,
	$U\subset W$ is a $G$-submodule. Moreover, with respect to the standard
	tensor basis of $W$, the submodule $U$ has a basis consisting of the sums
	of the elements in the $\mathfrak S_n$-orbits. Choosing one tensor in each
	orbit and mapping it to the sum of the elements in that orbit, while
	mapping the other tensors to zero, gives an $R$-linear retraction
	$W\to U$. Thus $U$ is an $R$-module direct summand of $W$ and $W/U$ is
	finite free.
	
Passing from these $G$-modules to the associated  vector bundles, we obtain
	\[
	F:=E(U)=\Gamma^n E,
	\qquad 
	A:=E(W)=E^{\otimes n}. 
	\]
By Corollary \ref{ample-tensor-product}, $A$ is ample. Choose a very ample line bundle $H$ on $X/R$
(recall that  by Proposition \ref{determinant}, $\det E$ is ample so such $H$ exists).  
Proposition \ref{equivalent-definitions} gives an integer $q>0$ such that 
	\[
	B:=\operatorname{Sym}^qA\otimes H^{-1}
	\]
	is globally generated. 
The inclusion $U\subset W$ induces a $G$-equivariant
	inclusion
	\[
	\operatorname{Sym}^qU\subset \operatorname{Sym}^qW.
	\]
	The group scheme $G=\GL(V)$ is reductive and  $U$ is an $R$-module direct summand of $W$, so we may apply
	Lemma~\ref{lem:power-retraction} to this inclusion.  For some $d>0$ it
	gives
	\[
	\Gamma^d(\operatorname{Sym}^qW)
	\longrightarrow
	\Sym^d(\Sym^qU),
	\]
	which sends $\gamma_d(u^q)$ to $(u^q)^d$.  Composing with multiplication of symmetric powers
$\Sym^d(\Sym^qU)\to	\Sym^{dq}U$, we get a $G$-equivariant map	
\[ \Gamma^d(\operatorname{Sym}^qW)	\longrightarrow	\Sym^{dq}U.\]
Passing to associated bundles and twisting by $H^{-d}$	gives
	\begin{equation}
		\label{eq:globalized-power-retraction}
		\beta\colon
		\Gamma^dB\simeq \Gamma^d(\operatorname{Sym}^qA)\otimes H^{-d}
		\longrightarrow
		\Sym^{dq}F\otimes H^{-d}.
	\end{equation}
Since $B$ is globally generated, the bundle  $\Gamma^dB$ is also globally generated.  Indeed, choose a finite
	surjection	$	\cO_X^N\twoheadrightarrow B.$
It splits locally because $B$ is locally free, and hence applying
	$\Gamma^d$ gives a surjection
	$\Gamma^d(\cO_X^N)\twoheadrightarrow\Gamma^dB$
from a trivial vector bundle.

Let $\pi\colon Y:=\PP_X(F)\longrightarrow X.$
Pulling back \eqref{eq:globalized-power-retraction} and composing with the
tautological quotient gives a homomorphism
\begin{equation}
	\label{eq:Q-quotient}
	\pi^*\Gamma^dB
	\longrightarrow
	Q:=\cO_Y(dq)\otimes\pi^*H^{-d}.
\end{equation}
This homomorphism is surjective. Indeed, it is enough to check this at
geometric points. Let $\overline{y}$ be a geometric point of $Y$, let
$\overline{x}=\pi(\overline{y})$, and let $\lambda\colon F_{\overline{x}}\twoheadrightarrow L$
be the corresponding one-dimensional quotient. Choose
$u\in F_{\overline{x}}$ with $\lambda(u)\neq 0$ and
$0\neq\eta\in H_{\overline{x}}^{-1}$. By the construction of $\beta$ and
its compatibility with base change, the element
\[
\gamma_d(u^q\otimes\eta)\in\Gamma^dB_{\overline{x}}
\]
maps to
\[
\lambda(u)^{dq}\eta^d
\in L^{dq}\otimes H_{\overline{x}}^{-d},
\]
which is nonzero. Thus \eqref{eq:Q-quotient} is everywhere surjective.
Since $\Gamma^dB$ is globally generated, $Q$ is globally generated.

Choose finitely many global sections generating $Q$, and let
$f\colon Y\longrightarrow\PP^m_R$
be the resulting morphism. Then
$f^*\cO_{\PP^m}(1)\simeq Q$. Consider
\[
h=(f,\pi)\colon Y\longrightarrow\PP^m_R\times_R X.
\]
Since $\pi$ is proper and $\PP^m_R\times_R X$ is separated over $X$, the
morphism $h$ is proper by \cite[Lemma 01W6]{Stacks}. For every geometric
point $\overline{x}\to X$, one has
\[
f_{\overline{x}}^*\cO_{\PP^m}(1)
\simeq
Q|_{Y_{\overline{x}}}
\simeq
\cO_{\PP(F_{\overline{x}})}(dq),
\]
which is ample. Therefore $f_{\overline{x}}$ has no
positive-dimensional fiber: on such a fiber the displayed line bundle
would be trivial, contradicting its ampleness. Hence the geometric fibers
of $h$ are zero-dimensional, so $h$ is quasi-finite. Since a proper
quasi-finite morphism is finite, $h$ is finite.

The line bundle $\cO_{\PP^m}(1)\boxtimes H^d$ is ample on
$\PP^m_R\times_R X$. Since $h$ is finite, it is affine, and finite
pullback preserves ampleness (see \cite[Lemma 0892]{Stacks}). Moreover,
\[
h^*\bigl(\cO_{\PP^m}(1)\boxtimes H^d\bigr)
\simeq
Q\otimes\pi^*H^d
\simeq
\cO_Y(dq).
\]
Thus $\cO_Y(dq)$ is ample. It follows that $\cO_Y(1)$ is ample, and
therefore $F=\Gamma^nE$ is ample.
\end{proof}

\medskip 
 
\section{Polynomial representations and relative ampleness}

In this section we show a few applications of our results. In the following we freely use the notions introduced in the Appendix.

\begin{corollary}\label{representations-of-ample-1}
	Let \(X\) be a scheme of finite type over a noetherian ring \(R\).
	Let \(E\) be an ample vector bundle of rank \(r\geq 1\) on \(X\), and
	let \(W\) be a finite locally free polynomial
	\(\GL_{r,R}\)-module of positive rank at every point of \(\Spec R\).
	If the degree zero homogeneous component \(W_0\) vanishes, then the
	associated vector bundle \(E(W)\) is ample. In particular,
	\(\Sym^n E\) and \(\Gamma^n E\) are ample for every \(n>0\), and
	\(\bigwedge^n E\) is ample for \(1\leq n\leq r\).
\end{corollary}
\begin{proof}
	Put $V=R^r$. By the homogeneous decomposition of $W$ and
	Proposition~\ref{prop:divided-power-generators}, there exist nonnegative
	integers $m_{d,\lambda}$, only finitely many of which are nonzero, and a
	$\GL_{r,R}$-equivariant surjection
	\[
	\bigoplus_{d>0}\ \bigoplus_{\lambda\in\Lambda^+(r,d)}
	\bigl(\Gamma^\lambda V\bigr)^{\oplus m_{d,\lambda}}
	\twoheadrightarrow W.
	\]
	The associated-bundle construction commutes with finite direct sums and
	tensor products and carries equivariant surjections to surjections.
	Moreover,
	\[
	E\bigl(\Gamma^\lambda V\bigr)
	\simeq
	\Gamma^\lambda E
	:=
	\bigotimes_{i=1}^r\Gamma^{\lambda_i}E.
	\]
	Consequently, the preceding surjection induces a surjection of vector
	bundles
	\[
	\bigoplus_{d>0}\ \bigoplus_{\lambda\in\Lambda^+(r,d)}
	\bigl(\Gamma^\lambda E\bigr)^{\oplus m_{d,\lambda}}
	\twoheadrightarrow E(W).
	\]
	
	By Theorem~\ref{thm:divided-powers-ample}, all the bundles
	$\Gamma^aE$, $a\geq1$, are ample. For
	$\lambda\in\Lambda^+(r,d)$ with $d>0$, the bundle $\Gamma^\lambda E$ is
	therefore ample by Corollary~\ref{ample-tensor-product}, where the factors
	corresponding to $\lambda_i=0$ are omitted. Proposition~\ref{direct-sums}
	then shows that the source of the displayed surjection is ample and that
	its quotient $E(W)$ is ample.

	Applying this to	$W=\Sym_R^n(R^r)$,	$W=\Gamma^n(R^r) $ and $W=\bigwedge_R^n(R^r)$,
	respectively, gives the final assertions.
\end{proof}

Using relative noetherian approximation, we obtain the following generalization of Corollary \ref{representations-of-ample-1}:

\begin{corollary} \label{representations-of-ample}
	Let \(f\colon X\to S\) be a morphism of schemes.
	\begin{enumerate}
		\item If \(E_1\) and \(E_2\) are \(f\)-ample vector bundles, then
		\(E_1\otimes E_2\) is \(f\)-ample.
		\item Let \(E\) be an \(f\)-ample vector bundle of rank \(r\geq1\),
		and let \(W\) be a finite locally free polynomial
		\(\GL_{r,S}\)-module of positive rank at every point of \(S\).
		If its degree-zero homogeneous component \(W_0\) vanishes, then
		\(E(W)\) is \(f\)-ample.
		
		\item If \(E\) is an \(f\)-ample vector bundle of rank \(r\), then
		\(\Sym^nE\) and \(\Gamma^nE\) are \(f\)-ample for every \(n>0\), and
		\(\bigwedge^nE\) is \(f\)-ample for \(1\leq n\leq r\).
	\end{enumerate}
\end{corollary}

\begin{proof}
Let \(\Spec R\subseteq S\) be an affine open subset.  The restrictions to \(X_R\) of the \(f\)-ample vector bundles under consideration are ample. Moreover, Lemma~\ref{relatively-ample-implies-qcqs} shows that \(X_R\) is
 quasi-compact and separated.
 So by relative noetherian approximation
 	\cite[Lemma~\texttt{0GS1}]{Stacks}, we may write
 	\[
 	\bigl(X_R\longrightarrow\Spec R\bigr)
 	\simeq
 	\varprojlim_{\lambda\in\Lambda}
 	\bigl(X_\lambda\longrightarrow S_\lambda\bigr),
 	\]
 	where all transition morphisms are affine and the schemes
 	\(X_\lambda\) and \(S_\lambda\) are of finite type over
 	\(\mathbb Z\). Since \(\Spec R\) is affine, after passing to a
 	cofinal subsystem we may assume that every \(S_\lambda\) is affine;
 	see \cite[Lemma~\texttt{01Z6}]{Stacks}. Write
 	$ S_\lambda=\Spec R_\lambda.$
 	Thus \(R_\lambda\) is noetherian and \(X_\lambda\) is of finite type
 	over \(R_\lambda\). Denote the canonical projections by
 	\[
 	g_\lambda\colon X_R\longrightarrow X_\lambda,
 	\qquad
 	h_\lambda\colon\Spec R\longrightarrow\Spec R_\lambda.
 	\]
 	
 	We first prove (1). Applying
 	Lemma~\ref{approximation-of-ample-vector-bundles} to
 	\((E_1)_R\) and \((E_2)_R\), and then passing to a common larger
 	index, we obtain ample vector bundles \(E_{1,\lambda}\) and
 	\(E_{2,\lambda}\) on \(X_\lambda\) such that
 	$
 	(E_i)_R\simeq g_\lambda^*E_{i,\lambda}$ for $ i=1,2.$
 	Since \(X_\lambda\) is of finite type over the noetherian ring
 	\(R_\lambda\), Corollary~\ref{ample-tensor-product} shows that
 	$	E_{1,\lambda}\otimes E_{2,\lambda}$
 	is ample. Consequently,
 	$(E_1\otimes E_2)_R
 	\simeq
 	g_\lambda^*
 	\bigl(E_{1,\lambda}\otimes E_{2,\lambda}\bigr)$
 	is ample by
 	Lemma~\ref{approximation-of-ample-vector-bundles}. Since this holds
 	over every affine open subset of \(S\), the vector bundle
 	\(E_1\otimes E_2\) is \(f\)-ample.
 	
 	We now prove (2). Write the homogeneous decomposition
 	$
 	W_R=\bigoplus_{d=1}^{N}W_{R,d}.$
 	The finite locally free modules \(W_{R,d}\), the homogeneous
 	coaction maps defining their polynomial
 	\(\GL_{r,R}\)-actions, and the identities satisfied by these maps
 	are finite-presentation data. Hence, after increasing \(\lambda\),
 	they descend by
 	\cite[Lemmas~\texttt{01ZR} and~\texttt{0B8W}]{Stacks} to a finite
 	locally free polynomial \(\GL_{r,R_\lambda}\)-module
 	\[
 	W_\lambda
 	=
 	\bigoplus_{d=1}^{N}W_{\lambda,d}
 	\]
 	such that 	$	W_R\simeq h_\lambda^*W_\lambda $ and 	$(W_\lambda)_0=0.$
 	At the same time,
 	Lemma~\ref{approximation-of-ample-vector-bundles} allows us to
 	arrange that	$	E_R\simeq g_\lambda^*E_\lambda$,
 	where \(E_\lambda\) is an ample vector bundle of rank \(r\) on
 	\(X_\lambda\). After increasing \(\lambda\) once more, we may also
 	assume that \(W_\lambda\) has positive rank at every point of
 	\(\Spec R_\lambda\); the assertions about the ranks follow from
 	\cite[Lemma~\texttt{05F4}]{Stacks}.
 	
 	Corollary~\ref{representations-of-ample-1}, applied over the
 	noetherian ring \(R_\lambda\), shows that
 	\(E_\lambda(W_\lambda)\) is ample. Since the associated-bundle
 	construction commutes with base change,
 	\[
 	E(W)|_{X_R}
 	\simeq
 	E_R(W_R)
 	\simeq
 	g_\lambda^*\bigl(E_\lambda(W_\lambda)\bigr).
 	\]
 	The last vector bundle is ample by
 	Lemma~\ref{approximation-of-ample-vector-bundles}. Since this holds
 	over every affine open subset of \(S\), \(E(W)\) is \(f\)-ample.
 	
 Finally, (3) follows from (2) by taking $W=\Sym_S^n(\cO_S^{\oplus r})$, $W=\Gamma^n(\cO_S^{\oplus r})$ and
 $W=\bigwedge\nolimits_S^n(\cO_S^{\oplus r})$,
 respectively.
 \end{proof}

\section{A nonample extension of ample bundles}

In this section we prove Theorem~\ref{extension-of-ample}. Then we give a related
example concerning the detection of ampleness from an open complement
and infinitesimal neighbourhoods of its boundary.

\begin{proof}[Proof of Theorem~\ref{extension-of-ample}]
	Let
	$ f\colon X:= \mathbb P_{\mathbb P^1} \bigl(\mathcal O_{\mathbb P^1}\oplus
	\mathcal O_{\mathbb P^1}(-2)\bigr) \longrightarrow\mathbb P^1$
	and let \(C\subset X\) be the negative section. Choose a nonzero
	extension
	\[
	0\longrightarrow\mathcal O_X(C)
	\longrightarrow G
	\longrightarrow\mathcal O_X
	\longrightarrow0.
	\]
	As observed in \cite[Section~4]{Ejiri-Fujino-Iwai},
	$G|_C\simeq\mathcal O_{\mathbb P^1}(-1)^{\oplus2}.$
	Let \(F\) be a fiber of \(f\) and let $s$ denote the canonical section of $\mathcal O_X(C)$. Let us set
	\[
	U:=X_s=X\setminus C,
	\qquad
	L:=\mathcal O_X(C+F)|_U,
	\qquad
	E:=\bigl(G\otimes\mathcal O_X(C+F)\bigr)|_U.
	\]
	Since \(\mathcal O_X(C)|_U\simeq\mathcal O_U\), restriction gives
	\[
	0\longrightarrow L\longrightarrow E\longrightarrow L\longrightarrow0.
	\]
	Moreover,
	\[
	U\simeq
	\mathbb V_{\mathbb P^1}\bigl(\mathcal O_{\mathbb P^1}(-2)\bigr),
	\qquad
	L\simeq(f|_U)^*\mathcal O_{\mathbb P^1}(1).
	\]
	Since \(f|_U\) is affine and pullback by an affine morphism
	preserves ampleness (see \cite[Lemma \texttt{0892}]{Stacks}), \(L\) is ample.
	
	We claim that \(E\) is not ample. Fix \(n\geq1\), and, for \(l\geq0\),
	set
	\[
	W_l:=
	\left(\operatorname{Sym}^nG\right)
	\otimes\mathcal O_X\bigl(lC+(n-1)F\bigr).
	\]
	The quotient \(G\twoheadrightarrow\mathcal O_X\) induces a morphism
	\[
	\sigma_l\colon
	W_l\longrightarrow
	\mathcal O_X\bigl(lC+(n-1)F\bigr).
	\]
	
	\begin{lemma}
		The map $H^0(\sigma _{l})$ vanishes for all $l\ge 0$.
	\end{lemma}
	
	\begin{proof}
		Since
		\[
		W_l|_C	\simeq
		\operatorname{Sym}^n
		\bigl(\mathcal O_{\mathbb P^1}(-1)^{\oplus2}\bigr)
		\otimes
		\mathcal O_{\mathbb P^1}(-2l+n-1)\simeq
		\mathcal O_{\mathbb P^1}(-2l-1)^{\oplus(n+1)}
		\]
		we have $H^0(C,W_l|_C)=0$ for $l\geq 0$.  Therefore the exact sequence
		\[
		0\longrightarrow W_{l-1}
		\mathop{\longrightarrow}^s W_l
		\longrightarrow W_l|_C
		\longrightarrow0
		\] 
		gives an isomorphism
		$H^0(X,W_{l-1})\xrightarrow{\ \sim\ }H^0(X,W_l)$ for \(l\geq1\).
		
		Let  us consider the commutative diagram
		$$\xymatrix{
			W_0\ar[rr]^{\sigma _0}\ar[d]^{\alpha_l}& & \cO_X((n-1)F)\ar[d]^{\beta_l}\\
			W_{l}\ar[rr]^-{\sigma _{l}}&& \cO_X(lC+(n-1)F)\\
		} $$ 
		in which the vertical maps are multiplication by $s^l$.  Since the map $H^0(\alpha_l)$
		is an isomorphism, it is sufficient to prove that $H^0(\sigma _{0})=0$.  
		
		Restriction to \(C\) gives a commutative diagram
		\[
		\xymatrix{
			H^0(X,W_0)\ar[r]^-{H^0(\sigma_0)}\ar[d]
			&
			H^0\bigl(X,\mathcal O_X((n-1)F)\bigr)\ar[d]\\
			H^0(C,W_0|_C)\ar[r]
			&
			H^0\bigl(C,\mathcal O_C((n-1)F)\bigr).
		}
		\]
		The right vertical arrow is an isomorphism because
		\(f|_C\colon C\to\mathbb P^1\) is an isomorphism and $$H^0\bigl(X,\mathcal O_X((n-1)F)\bigr)=H^0\bigl(\PP^1,\mathcal O_{\PP^1}(n-1))$$
		by $f_*\cO_X=\cO_{\PP^1}$.
		Since $ H^0(C,W_0|_C)=0$ we have \(H^0(\sigma_0)=0\), and therefore \(H^0(\sigma_l)=0\) for every
		\(l\geq0\).
	\end{proof}
	
	Put $\overline E:=G\otimes\mathcal O_X(C+F)$ and
	$\overline L:=\mathcal O_X(C+F)$. For every $n\geq1$, set
	\[
	\mathcal A_n:=
	\operatorname{Sym}^n\overline E\otimes\overline L^{-1},
	\qquad
	\mathcal B_n:=\overline L^{n-1},
	\]
	and let
	\[
	\rho_n\colon\mathcal A_n\twoheadrightarrow\mathcal B_n
	\]
	be the quotient induced by
	$\operatorname{Sym}^n\overline E\twoheadrightarrow\overline L^n$.
	For every $m\geq0$, we have
	\[
	\mathcal A_n(mC)
	\simeq
	\operatorname{Sym}^nG\otimes
	\mathcal O_X\bigl((n-1+m)C+(n-1)F\bigr)
	=W_{n+m-1}
	\]
	and
	\[
	\mathcal B_n(mC)
	\simeq
	\mathcal O_X\bigl((n-1+m)C+(n-1)F\bigr).
	\]
	Under these identifications,
	$\rho_n\otimes\mathcal O_X(mC)$ is the map
	$\sigma_{n+m-1}$. Hence the lemma shows that it induces the zero map
	on global sections for every $m\geq0$.
	
 Since	$U=X_{s}$ and $X$ is quasi-compact and quasi-separated,
	\cite[Lemma \texttt{01PW}]{Stacks} gives, functorially in every
	quasi-coherent sheaf $\mathcal H$ on $X$,
	\[
	H^0(U,\mathcal H|_U)
	\simeq
	\varinjlim_{m\geq0}H^0\bigl(X,\mathcal H(mC)\bigr),
	\]
	where the transition maps are multiplication by $s$. Applying this
	to $\mathcal A_n$ and $\mathcal B_n$ and passing to the direct limit,
	we obtain that
	\[
	H^0\bigl(U,\operatorname{Sym}^nE\otimes L^{-1}\bigr)
	\longrightarrow
	H^0(U,L^{n-1})
	\]
	is zero for every $n\geq1$.
	
	If $\operatorname{Sym}^nE\otimes L^{-1}$ were globally generated, then
	the composition
	\[
	H^0\bigl(U,\operatorname{Sym}^nE\otimes L^{-1}\bigr)
	\otimes\mathcal O_U
	\longrightarrow
	\operatorname{Sym}^nE\otimes L^{-1}
	\twoheadrightarrow
	L^{n-1} 
	\]
	would be surjective. But this composition is zero by the preceding
	paragraph, a contradiction. Thus
	$\operatorname{Sym}^nE\otimes L^{-1}$ is not globally generated for
	any $n\geq1$. So by
	Proposition~\ref{equivalent-definitions}, we conclude that $E$ is not
	ample. Since \(X\) is a smooth projective surface, \(U\) is smooth quasi-projective, so taking \(S:=U\) proves the theorem.
\end{proof}

\medskip

\begin{example}\label{example:funny}
	The following example shows that ampleness on a nonproper variety
	cannot be detected from the complement of a divisor together with all
	the infinitesimal neighbourhoods of that divisor.	
	
	Let us consider $S$, $E$ and $L$ as in Theorem \(\ref{extension-of-ample}\) and set
	$ \pi:Y=\mathbf P_S(E)\to S$ and $ A=\mathcal O_Y(1). $
	Although \(\pi\) is projective, \(Y\) is a smooth nonproper threefold. The quotient \(E\twoheadrightarrow L\) determines a smooth Cartier divisor \(D\simeq S\), and
	$ A|_D\simeq L $
	is ample. If \(U=Y\setminus D\), the inclusion \(L\hookrightarrow E\) gives
	$ A|_U\simeq (\pi|_U)^*L. $
	Since \(U\to S\) is an affine \(\mathbf A^1\)-bundle, \(A|_U\) is ample by  \cite[Lemma \texttt{0892}]{Stacks}.
	Moreover, if \(D_n\) is any infinitesimal neighbourhood of \(D\), then
	$ A|_{D_n}$ is ample for every $n$, 
	because ampleness is unchanged by nilpotent thickening. Nevertheless \(A\) is not ample on \(Y\), since \(E\) is not ample.
	
	Thus, on a smooth nonproper variety $Y$,  ampleness of \(A\) cannot
	be recovered even from all of the following:
	\begin{enumerate}
		\item 
		ampleness on the open complement $Y\setminus D$,
		\item 
		ampleness on the smooth boundary $D$, 
		\item 
		ampleness on every finite-order infinitesimal thickening of that boundary.
	\end{enumerate}
\end{example}

\appendix\label{sec:polynomial-quotients}

\section{Polynomial representations of general linear groups}

In this appendix we recall the facts about polynomial representations of
general linear groups that are used in the paper. Throughout, \(R\) is a
ring, \(V=R^r\), and \(G=\operatorname{GL}(V)\). This is
sufficient for our relative applications: after passing to a suitable
affine open cover of the base, the relevant vector bundles become free,
and the constructions below are compatible with localization. Thus the
resulting statements apply locally on the base to polynomial
\(\GL_{r,S}\)-modules over an arbitrary scheme \(S\).

\medskip

Let $M$ be a finite $R$-module. Let us recall (see \cite[Part II, Chapter A]{Jantzen}) that a $G$-module $M$ is
\emph{polynomial} if its $G$-action extends to an action of the algebraic
monoid $\operatorname{End}_R(V)$. Equivalently, the corresponding coaction
is a homomorphism
\[
M\longrightarrow
M\otimes_R\operatorname{Sym}_R\bigl(\operatorname{End}_R(V)^*\bigr).
\]
A polynomial $G$-module $M$ is \emph{homogeneous of degree $d$} if this
coaction has image in
\[
M\otimes_R
\operatorname{Sym}_R^d\bigl(\operatorname{End}_R(V)^*\bigr).
\]
If $M$ is finite locally free, then, after choosing bases locally, this is
equivalent to requiring that the matrix coefficients of the representation
be homogeneous polynomials of degree $d$ in the matrix entries.

For \(d\geq 0\), we  denote by
\(\mathcal P_R(r,d)\) the category of finite homogeneous polynomial
\(G\)-modules of degree \(d\), and let
\[
\Lambda^+(r,d)
:=
\left\{
\lambda=(\lambda_1,\ldots,\lambda_r)\in\mathbb Z_{\geq 0}^r
\ \middle|\
\lambda_1\geq\lambda_2\geq\cdots\geq\lambda_r
\ \text{and}\
\sum_{i=1}^r\lambda_i=d
\right\}
\]
denote the set of partitions of \(d\) with at most \(r\) nonzero
parts. We allow trailing zeroes and use the convention
\(\Gamma^0V=R\). For \(\lambda\in\Lambda^+(r,d)\) we write
\[
\Gamma^\lambda V
:=
\bigotimes_{i=1}^r\Gamma^{\lambda_i}V
\]
where the divided powers $\Gamma^{\lambda_i}V$ were defined in Subsection \ref{subsection:divided-powers}. 

The following proposition is the characteristic-free form of the
observation used by Totaro \cite[pp.~2 and 5]{Totaro}; the Schur-algebra
description used there goes back to Akin and Buchsbaum
\cite{Akin-Buchsbaum}.

\begin{proposition}
	\label{prop:divided-power-generators}
	Let \(M\in\mathcal P_R(r,d)\). Then there exist nonnegative integers
	\(m_\lambda\), indexed by \(\lambda\in \Lambda^+(r,d)\), 
	such that we have a \(G\)-equivariant surjection
	\[
	\bigoplus_{\lambda\in\Lambda^+(r,d)}
	\bigl(\Gamma^\lambda V\bigr)^{\oplus m_\lambda}
	\twoheadrightarrow M.
	\]
\end{proposition}

\begin{proof}
	Set
	\[
	C_R(r,d):=
	\operatorname{Sym}_R^d\bigl(\operatorname{End}_R(V)^*\bigr).
	\]
	This is a finite free $R$-coalgebra. Hence the category
	$\mathcal P_R(r,d)$ is equivalent to the category of finite left modules
	over the Schur algebra
	\[
	S_R(r,d):=C_R(r,d)^*
	\simeq\Gamma^d\bigl(\operatorname{End}_R(V)\bigr).
	\]
	
	Consider $S_R(r,d)$ as a left module over itself. The corresponding
	$G$-action on 
	$\Gamma^d(\operatorname{End}_R(V))$ is induced by left multiplication on
	$\operatorname{End}_R(V)$. After choosing a basis of $V$, we have an
	isomorphism of $G$-modules
	\[
	\operatorname{End}_R(V)
	=V\otimes_RV^*
	\simeq V^{\oplus r}.
	\]
	The divided-power decomposition of a direct sum therefore gives
	\[
S_R(r,d)\simeq \Gamma^d\bigl(V^{\oplus r}\bigr)\simeq
		\bigoplus_{\substack{\alpha=(\alpha_1,\ldots,\alpha_r)
				\in\mathbb Z_{\geq0}^r\\
				\alpha_1+\cdots+\alpha_r=d}}
		\Gamma^{\alpha_1}V\otimes_R\cdots\otimes_R\Gamma^{\alpha_r}V.
	\]
	After reordering the tensor factors, every summand in this decomposition
	is isomorphic to $\Gamma^\lambda V$ for a unique
	$\lambda\in\Lambda^+(r,d)$. Consequently, there exist positive integers
	$c_\lambda$ such that
	\[
	S_R(r,d)\simeq
	\bigoplus_{\lambda\in\Lambda^+(r,d)}
	\bigl(\Gamma^\lambda V\bigr)^{\oplus c_\lambda}
	\]
	as polynomial $G$-modules.
	
	Since $M$ is finite as an $R$-module, it is finitely generated as an
	$S_R(r,d)$-module. Hence, for some $N\geq0$, there is an
	$S_R(r,d)$-linear surjection
	\[
	S_R(r,d)^{\oplus N}\twoheadrightarrow M.
	\]
	Using the preceding decomposition of the regular Schur module gives the
	required $G$-equivariant surjection.
\end{proof}

\begin{corollary}
	\label{cor:classical-operations-generate}
	Let \(\mathcal C(V)\) be the smallest class of finite polynomial
	\(G\)-modules containing all divided powers \(\Gamma^nV\), \(n\geq1\),
	and closed under finite direct sums, tensor products, quotients, and
	isomorphisms. Then the nonzero objects of \(\mathcal C(V)\) are precisely
	the nonzero finite polynomial \(G\)-modules whose homogeneous component
	of degree zero vanishes.
\end{corollary}

\begin{proof}
	Every finite polynomial $G$-module has a canonical decomposition
	\[
	M=\bigoplus_{d\geq0}M_d
	\]
	into homogeneous polynomial $G$-modules, where the scalar matrix
	$t\operatorname{id}_V$ acts on $M_d$ as multiplication by $t^d$. Only
	finitely many of the $M_d$ are nonzero, and $M_0$ is a trivial
	$G$-module.
	
	Suppose first that $M_0=0$. By
	Proposition~\ref{prop:divided-power-generators}, every homogeneous
	component $M_d$, $d>0$, is a quotient of a finite direct sum of modules
	$\Gamma^\lambda V$. Every such module belongs to $\mathcal C(V)$, since
	in its definition one may omit the factors $\Gamma^{\lambda_i}V$ for
	which $\lambda_i=0$. Hence $M$ belongs to $\mathcal C(V)$.
	
	Conversely, every generator $\Gamma^nV$, $n\geq1$, is homogeneous of
	positive degree. Finite direct sums, tensor products, quotients, and
	isomorphisms preserve the vanishing of the degree-zero component.
	Therefore every nonzero object of $\mathcal C(V)$ has $M_0=0$.
\end{proof}

If $M$ is finite locally free over $R$, the condition $M_0=0$ is
equivalent to requiring that, for every geometric point
$\bar{s}\to\Spec R$, the polynomial
$\GL_{r,\kappa(\bar{s})}$-module $M_{\bar{s}}$ has no trivial
composition factor.

\section*{Declaration on the use of generative AI} 

During the development and preparation of this article, the author used ChatGPT (OpenAI; GPT-5.6 Sol, accessed August 2026) extensively as an interactive tool to explore proof strategies, including the arguments concerning divided powers, the counterexample, and the extension to arbitrary bases, and to assist with drafting and revision. The author independently checked all mathematical arguments and references, rewrote the generated material, and assumes full responsibility for the contents of the article.

\section*{Acknowledgements}

The author was partially supported by the National Science Centre, Poland, contract number 2025/59/B/ST1/02168.

\bibliographystyle{amsalpha}
\bibliography{References}

\providecommand{\bysame}{\leavevmode\hbox to3em{\hrulefill}\thinspace}
\providecommand{\MR}{\relax\ifhmode\unskip\space\fi MR }
\providecommand{\MRhref}[2]{%
  \href{http://www.ams.org/mathscinet-getitem?mr=#1}{#2}
}
\providecommand{\href}[2]{#2}
\begin{thebibliography}{FvdK10}

\bibitem[AB88]{Akin-Buchsbaum}
Kaan Akin and David~A. Buchsbaum, \emph{Characteristic-free representation
  theory of the general linear group. {II}. {H}omological considerations},
  Advances in Mathematics \textbf{72} (1988), no.~2, 171--210.

\bibitem[Bar71]{Barton}
Charles~M. Barton, \emph{Tensor products of ample vector bundles in
  characteristic {$p$}}, American Journal of Mathematics \textbf{93} (1971),
  no.~2, 429--438.

\bibitem[Ben10]{Benoist-relative-ampleness}
Olivier Benoist, \emph{Answer to ``relatively ample line bundles''},
  MathOverflow, 2010, \url{https://mathoverflow.net/a/28018} Posted June 13,
  2010.

\bibitem[DCP76]{DeConcini-Procesi}
Corrado De~Concini and Claudio Procesi, \emph{A characteristic free approach to
  invariant theory}, Advances in Mathematics \textbf{21} (1976), no.~3,
  330--354.

\bibitem[DF04]{DomokosFrenkel}
M{\'a}ty{\'a}s Domokos and P{\'e}ter~E. Frenkel, \emph{On orthogonal invariants
  in characteristic 2}, Journal of Algebra \textbf{274} (2004), no.~2,
  662--688.

\bibitem[EFI24]{Ejiri-Fujino-Iwai}
Sho Ejiri, Osamu Fujino, and Masataka Iwai, \emph{Positivity of extensions of
  vector bundles}, Mathematische Zeitschrift \textbf{306} (2024), Paper No. 47,
  8 pp.

\bibitem[FvdK10]{FvdK}
Vincent Franjou and Wilberd van~der Kallen, \emph{Power reductivity over an
  arbitrary base}, Documenta Mathematica \textbf{Extra Volume} (2010),
  171--195, Andrei A. Suslin's sixtieth birthday.

\bibitem[Gro61]{EGA2}
Alexander Grothendieck, \emph{{\'E}l\'ements de g\'eom\'etrie alg\'ebrique.
  {II}. {{\'E}}tude globale \'el\'ementaire de quelques classes de morphismes},
  Publications Math\'ematiques de l'IH\'ES \textbf{8} (1961), 5--222,
  R\'edig\'e avec la collaboration de Jean Dieudonn\'e.

\bibitem[Har66]{Hartshorne-Ample}
Robin Hartshorne, \emph{Ample vector bundles}, Publications Math{\'e}matiques
  de l'Institut des Hautes {\'E}tudes Scientifiques \textbf{29} (1966), 63--94.

\bibitem[Has05]{Hashimoto}
Mitsuyasu Hashimoto, \emph{Another proof of theorems of {De Concini} and
  {Procesi}}, Journal of Mathematics of Kyoto University \textbf{45} (2005),
  no.~4, 701--710.

\bibitem[Jan03]{Jantzen}
Jens~Carsten Jantzen, \emph{Representations of algebraic groups}, 2nd ed.,
  Mathematical Surveys and Monographs, vol. 107, American Mathematical Society,
  Providence, RI, 2003.

\bibitem[Lan21]{Langer-Moduli-Lie-algebroids}
Adrian Langer, \emph{Moduli spaces of semistable modules over {L}ie
  algebroids}, 2021, \url{https://arxiv.org/abs/2107.03128}.

\bibitem[MFK94]{Mumford-Fogarty-Kirwan}
David Mumford, John Fogarty, and Frances Kirwan, \emph{Geometric invariant
  theory}, 3rd ed., Ergebnisse der Mathematik und ihrer Grenzgebiete. 2. Folge,
  vol.~34, Springer-Verlag, Berlin, 1994.

\bibitem[Ser68]{Ser68}
Jean-Pierre Serre, \emph{Groupe de {Grothendieck} des sch\'emas en groupes
  r\'eductifs d\'eploy\'es}, Publications Math\'ematiques de l'IH\'ES
  \textbf{34} (1968), 37--52.

\bibitem[Ses77]{Seshadri-Geometric-Reductivity}
C.~S. Seshadri, \emph{Geometric reductivity over arbitrary base}, Advances in
  Mathematics \textbf{26} (1977), no.~3, 225--274.

\bibitem[{Sta}26]{Stacks}
The {Stacks Project Authors}, \emph{\textit{Stacks Project}},
  \url{https://stacks.math.columbia.edu}, 2026.

\bibitem[Tot97]{Totaro}
Burt Totaro, \emph{Projective resolutions of representations of
  {$\mathop{GL}(n)$}}, Journal f{\"u}r die reine und angewandte Mathematik
  \textbf{482} (1997), 1--13.

\bibitem[vdK21]{vdK}
Wilberd van~der Kallen, \emph{Reductivity properties over an affine base},
  Indagationes Mathematicae \textbf{32} (2021), no.~5, 961--967, Special issue
  in memory of T. A. Springer.

\end{thebibliography}

\end{document}